\documentclass[12pt]{vlad}
\usepackage[english]{babel}

\usepackage[fontsize=12pt]{scrextend}

\theoremstyle{plain}
\newtheorem{theorem}{Theorem}
\newtheorem{lemma}{Lemma}

\title[Runs of Consecutive Integers Having the Same Number of Divisors]{Runs of Consecutive Integers Having the Same Number of Divisors}
\address{V. T. Sp\u{a}taru \\
Bucharest, Romania}
\email{vtspataru@gmail.com} 
\author[Vlad-Titus Sp\u{a}taru]{\large Vlad-Titus Sp\u{a}taru}
\urladdr{}
\givenname{Vlad-Titus}
\surname{Sp\u{a}taru}

\subject{primary}{matsc2000}{11A25}
\subject{primary}{matsc2000}{11N37}

\volumenumber{6}
\publicationyear{2023}
\startpage{96}

\begin{document}

\begin{abstract}
Our objective is to provide an upper bound for the length $\ell_N$ of the longest run of consecutive integers smaller than $N$ which have the same number of divisors. We prove in an elementary way that $\log\ell_N\ll(\log N\log\log N)^\lambda$, where $\lambda=1/2$. Using estimates for the Jacobsthal function, we then improve the result to $\lambda=1/3$.
\end{abstract}

\maketitle

\section{Introduction}

Let $d(n)$ denote the number of positive divisors of $n$. The equation $d(n)=d(n+k)$ has been studied extensively. Spiro \cite{SP} showed that it has infinitely many solutions for $k=5040$. Subsequently, Heath-Brown \cite{HB} established the case $k=1$, and Pinner \cite{PN} ultimately proved that all values of $k$ yield infinitely many solutions.

As $d(n)$ is equal to $d(n+1)$ infinitely often, one naturally wonders \emph{how many} consecutive integers can there be, having the same number of divisors. Erd\H{o}s and Mirsky \cite{ER} conjectured that there are arbitrarily long such runs of integers. They were not able to provide any estimates for the length of such sequences: {\it``A related problem consists in the estimation of the longest run of consecutive integers $\leqslant x$ all of which have the same number of divisors. This problem seems to be one of exceptional difficulty, and we} [Erd\H{o}s \& Mirsky] \textit{have not been able to make any progress with it.''}

Our principal objective is to provide an upper bound for the length $\ell_N$ of the runs in question. In Section \ref{sec:2}, we estimate the order of magnitude of the $d(n)$ and $\omega(n)$ functions and obtain the following estimate in an elementary manner: 

\begin{theorem}\label{thm:1}
Let $\ell_N$ denote the length of the longest run of consecutive integers smaller than $N$, having the same number of divisors. Then, \[\log\ell_N\ll\sqrt{\log N\cdot \log\log N}.\]
\end{theorem}

Subsequently, in Section \ref{sec:3}, we provide a quicker proof of Theorem \ref{thm:1} based on the Prime Number Theorem. We also get the following explicit form of Theorem \ref{thm:1}: \[\log\ell_N\leqslant\sqrt{(1/2+o(1))\log N\log\log N}.\]

Finally, in Section \ref{sec:4} we prove a stronger version of Theorem \ref{thm:1} using estimates of the Jacobsthal function deduced from Brun's sieve method. We obtain the following result:

\begin{theorem}\label{thm:2}
Let $\ell_N$ denote the length of the longest run of consecutive integers smaller than $N$, having the same number of divisors. Then, \[\log\ell_N\ll\sqrt[3]{\log N\cdot \log\log N}.\]
\end{theorem}

\section{An Elementary Proof of Theorem \ref{thm:1}}\label{sec:2}

In proving Theorem \ref{thm:1}, we will make use of the following lemmas, the first being proven in an elementary manner in \cite{FA} and the second being Mertens' bound.

\begin{lemma}\label{lem:1}
Let $n$ be a positive integer. Then, $\lcm(1,2,\ldots,n+1)\geqslant 2^n$.
\end{lemma}

\begin{lemma}\label{lem:2}The sum of the reciprocals of the prime numbers not exceeding $n$ satisfies
\[\sum_{p\leqslant n}\frac{1}{p}=\log\log n+M+O\left(\frac{1}{\log n}\right)\ll\log\log n.\]
\end{lemma}

Note that it suffices to prove that Theorem \ref{thm:1} holds for large enough $N$. Assume that there exist $k>2$ consecutive numbers smaller than $N$, having the same number of divisors. Let them be $n+1$, $n+2$, \ldots, $n+k$ and write \[d(n+1)=d(n+2)=\cdots=d(n+k)=D.\] We will firstly provide an estimate for $D$, in terms of $k$. For simplicity, let $K=\lfloor\log_2 k\rfloor$. 

As $k\geqslant 2^K$, all residues modulo $2^K$ are among $n+1$, $n+2$, \ldots, $n+k$. Therefore, for all $1\leqslant i\leqslant K-1$, there exists some $1\leqslant t_i\leqslant k$ such that $n+t_i\equiv 2^{i}\bmod{2^K}$. Consequently, $\nu_2(n+t_i)=i$, so $i+1$ divides $d(n+t_i)=D$. 

Hence, $D$ is divisible by $\lcm(1,2,\ldots, K)$. Using Lemma \ref{lem:1}, we infer that \[D\geqslant \lcm(1,2,\ldots, K)\geqslant 2^{K-1}.\] Recall that $K=\lfloor\log_2 k\rfloor\geqslant \log_2k-1$, so $D\geqslant k/4$.

Let $\omega(n)$ denote the number of distinct prime factors of $n$. Choose $1\leqslant l\leqslant k$ arbitrarily. As $n+l\leqslant N$, it follows that $\nu_p(n+l)\leqslant\log_p N\leqslant \log_2 N$ for all prime numbers $p$. Therefore, \[D=d(n+l)=\prod_p(\nu_p(n+l)+1)\leqslant\prod_{p\mid n+l}(\log_2 N+1)=(\log_2 N+1)^{\omega(n+l)}.\] 

Taking logarithms, it follows that $\omega(n+l)\geqslant \log D/\log(\log_2 N+1)$. Moreover, note that a prime number $p$ can divide at most $\lceil k/p\rceil\leqslant k/p+1$ numbers among $n+1$, \ldots, $n+k$. Therefore, using Lemma \ref{lem:2}, it follows that \[\omega((n+1)\cdots(n+k))\geqslant\sum_{i=1}^k\omega(n+i)-\sum_{p\leqslant k}\frac{k}{p}\geqslant\frac{k\log D}{\log(\log_2 N+1)}-C_1k\log\log k,\] 
for a suitable constant $C_1$. Further, we will write $\log(\log_2 N+1)\leqslant C_2\log\log N$, for some constant $C_2$. Recall that $D\geqslant k/4$, so we have \begin{equation}\label{eq:1}\omega((n+1)\cdots(n+k))\geqslant\frac{k\log (k/4)}{C_2\log\log N}-C_1k\log\log k.\end{equation}

Write the right-hand side of equation \ref{eq:1} as $k\cdot f_N(k)$. Clearly, if $\omega(a)\geqslant b$ then $a\geqslant b!$. Using this remark on equation \ref{eq:1}, we get $(n+1)\cdots(n+k)\geqslant\lceil k\cdot f_N(k)\rceil!$. Moreover, because $N^k\geqslant (n+1)\cdots(n+k)$, by applying the well-known inequality $\log t!\geqslant t\log t-t$, we have \begin{align}\label{eq:2}k\log N &\geqslant \log\left((n+1)\cdots (n+k)\right)\geqslant \log\left(\lceil k\cdot f_N(k)\rceil!\right) \nonumber \\ &\geqslant k\cdot f_N(k)\cdot\log(k\cdot f_N(k))-k\cdot f_N(k).\end{align} 

Finally, dividing equation \ref{eq:2} by $k$ we obtain \begin{equation}\label{eq:3}\log N\geqslant f_N(k)\cdot\log(k\cdot f_N(k))- f_N(k).\end{equation}

Define the interval $I_N=\left[\exp\left(C_1\cdot C_2\cdot\log\log N\right),\infty\right)$. Using standard arguments, one may infer that $f_N$ is increasing on $I_N$.

Let us suppose, for the sake of contradiction, that $k>\exp\left(C\sqrt{\log N \log\log N}\right)$, where $C>\max(\sqrt{C_2},C_1\cdot C_2)$. Firstly, note that since $\log N>\log\log N$ and $C>C_1\cdot C_2$ then $\exp\left(C\sqrt{\log N \log\log N}\right)$ and $k$ are in $I_N$. Therefore, we have \begin{align}\label{eq:4}f_N(k)&>f_N\left(\exp\left(C\sqrt{\log N \log\log N}\right)\right) \nonumber\\ &=\frac{C}{C_2}\sqrt{\frac{\log N}{\log\log N}}-\frac{\log 4}{C_2\log\log N}-C_1\log\left(C\sqrt{\log N\log\log N}\right).\end{align} Viewing equation \ref{eq:4} as a function in $N$, it is evident that for large enough $N$ (greater than some $N_1$) we also have $f_N(k)>e$. In what follows, we will assume that $N>N_1$.

As $f_N(k)>e$, it follows from equation \ref{eq:3} that $\log N\geqslant f_N(k)\cdot\log k$. Further, applying equation \ref{eq:4} and the estimate for $k$ and isolating the term $\log N$, we get \[\frac{C\log 4}{C_2}\sqrt{\frac{\log N}{\log\log N}}+C_1C\sqrt{\log N\log\log N}\log\left( C\sqrt{\log N\log\log N}\right)\geqslant\left(\frac{C^2}{C_2}-1\right)\log N.\] Recall that $C>\sqrt{C_2}$, so the latter inequality is absurd for large enough $N$ (greater than some $N_2$), as the left-hand side is asymptotically much smaller than $\log N$. Therefore, Theorem \ref{thm:1} holds for $N>\max(N_1,N_2)$ and $C>\max(\sqrt{C_2},C_1\cdot C_2)$.

\section{An Explicit Form of Theorem \ref{thm:1}}\label{sec:3}

Let $\Omega(n)$ denote the number of prime factors of $n$, counting multiplicities. Note that $2^{\omega(n)}\leqslant d(n)\leqslant 2^{\Omega(n)}$ for all positive integers $n$. Further, we have the following estimates: 

\begin{lemma}\label{lem:3}The product $n\#$ of the prime numbers not exceeding $n$ satisfies $\log n\#\sim n$.
\end{lemma}

\begin{lemma}\label{lem:4}The sum of $1/\log p$ taken over the prime numbers not exceeding $n$ satisfies
\[\sum_{p\leqslant n}\frac{1}{\log p}=\frac{n}{(\log n)^2}+O\left(\frac{n\log\log n}{(\log n)^3}\right)=o(1)\cdot\frac{n}{\log n}.\]
\end{lemma}

Let $M$ be the greatest positive integer satisfying $M\#\leqslant k$. Then, there exists $1\leqslant m\leqslant k$ so that $M\#$ divides $n+m$. Hence, for any $1\leqslant i\leqslant k$ we have \[2^{\Omega(n+i)}\geqslant d(n+i)=d(n+m)\geqslant d(M\#)\geqslant 2^{\omega(M\#)}=2^{\pi(M)},\]so $\Omega(n+i)\geqslant\pi(M)$. It follows that $\Omega(n+1)+\Omega(n+2)+\cdots+\Omega(n+k)\geqslant k\pi(M)$.

Fix a prime number $p\leqslant k$. Then, for every exponent $t\leqslant \log_pN$ there are at most $\lceil k/p^t\rceil\leqslant k/p^t+1$ numbers divisible by $p^t$ among $n+1$, $n+2$, \ldots, $n+k$. Furthermore, if $t> \log_pN$, then $p^t>N$ so none of $n+1$, $n+2$, \ldots, $n+k$ are divisible by $p^t$.

Following the same double-counting technique as in Legendre's Theorem, we may infer \begin{equation}\label{eq:5}\sum_{i=1}^k\nu_p(n+i)\leqslant\sum_{t=1}^{\lfloor\log_pN\rfloor}\left(1+\frac{k}{p^t}\right)<\frac{\log N}{\log p}+\frac{k}{p-1}.\end{equation}

Now, observe that $n+i$ has at most $\log_k N$ prime factors greater than $k$, including multiplicities. Combining this observation with equation \ref{eq:5} and Lemmas \ref{lem:2} and \ref{lem:4}, we get \begin{align}\label{eq:6}k\pi(M)&\leqslant\sum_{i=1}^k\Omega(n+i)=\sum_{p>k}\sum_{i=1}^k\nu_p(n+i)+\sum_{p\leqslant k}\sum_{i=1}^k\nu_p(n+i) \nonumber \\ &\leqslant\frac{k\log N}{\log k}+\sum_{p\leqslant k}\left(\frac{\log N}{\log p}+\frac{k}{p-1}\right)=(1+o(1))\frac{k\log N}{\log k}+O(k\log\log k).\end{align}

It follows from Lemma \ref{lem:3} and the Prime Number Theorem that $\pi(M)\sim\log k/\log\log k$. Further, note that $k\log\log k= o(1)\cdot k\log k/\log\log k$. Using these observations in equation \ref{eq:6} we get \[\frac{k\log k}{\log\log k}\sim k\pi(M)\leqslant (1+o(1))\frac{k\log N}{\log k}+o(1)\frac{k\log k}{\log\log k}.\]

Dividing through $k$ and isolating the remaining functions in $k$ on the left-hand side, we have $(\log k)^2/\log\log k\leqslant (1+o(1))\log N.$ Using standard arguments, we finally get \[\log k\leqslant\sqrt{(1/2+o(1))\log N\log\log N}.\]

\section{The Proof of Theorem \ref{thm:2}}\label{sec:4}

We will keep the notation used in Section \ref{sec:2}. We will require the following estimates:

\begin{lemma}\label{lem:5}
The sum of the prime numbers not exceeding $n$ satisfies \[\sum_{p\leqslant n}p=\frac{n^2}{2\log n}+O\left(\frac{n^2}{(\log n)^2}\right)\sim\frac{n^2}{2\log n}.\]
\end{lemma}

\begin{lemma}\label{lem:6}
Let $n$ be a positive integer and $p_{\text{min}}$ be its smallest prime divisor. Then, \[\frac{\log n}{\log p_{\text{min}}}\geqslant\sum_{p}(p-1)\nu_p(d(n)).\]
\end{lemma}

\begin{proof}
Throughout the rest of the proof, the letters $p$ and $q$ will refer strictly to prime numbers. Note that since $p_{\text{min}}$ is the smallest prime factor of $n$, by taking logarithms we get\begin{equation}\label{eq:7}\log n=\sum_q \nu_q(n)\log q\geqslant \log p_{\text{min}}\sum_q\nu_q(n).\end{equation}Further, using the fact that $m^n-1\geqslant n(m-1)$ for all positive integers, we infer that \begin{equation}\label{eq:8}k=\prod_{p}p^{\nu_p(k+1)}-1\geqslant\sum_p \left(p^{\nu_p(k+1)}-1\right)\geqslant\sum_p (p-1)\nu_p(k+1)\end{equation}for any positive integer $k$. Using inequality \ref{eq:8} on $\nu_q(n)$ in equation \ref{eq:7}, we further have \begin{align*}\frac{\log n}{\log p_{\text{min}}}&\geqslant\sum_q \nu_q(n)\geqslant\sum_q\sum_p (p-1)\nu_p(\nu_q(n)+1)=\sum_p\left((p-1)\sum_q\nu_p(\nu_q(n)+1)\right) \\ &=\sum_p(p-1)\nu_p\left(\prod_q(\nu_q(n)+1)\right)=\sum_p (p-1)\nu_p(d(n)),\end{align*}giving us the desired result.
\end{proof}

The proof now hinges on finding an index $i$ for which $n+i$ has a large minimal prime factor. Jacobsthal \cite{JA} defines the function $j(n)$  to be the least integer so that amongst any $j(n)$ consecutive integers there exists at least one relatively prime to $n$. 

Therefore, if a positive integer $M$ satisfies $j(M\#)\leqslant k$, then some $n+i$ has the minimal prime factor larger than $M$. As pointed out by Erd\H{o}s \cite{ER2}, it follows directly from Brun's sieve that $j(n)\ll \omega(n)^C$ for a suitable constant $C$, hence $\log j(n)\ll \log \omega(n)$.

\newpage

It follows that $\log j(M\#)\ll\log\omega(M\#)=\log\pi(M)\ll\log M$, so among our $k$ consecutive integers we may find one, $n+i$, whose minimal prime factor $p_{\text{min}}$ satisfies $\log p_{\text{min}}\gg \log k$. Further, recall that every prime number not exceeding $\log_2k$ divides $D=d(n+i)$. Applying Lemmas \ref{lem:5} and \ref{lem:6} we then get \[\label{eg:9}\frac{\log N}{\log k}\gg\frac{\log (n+i)}{\log p_{\text{min}}}\geqslant \sum_{p}(p-1)\nu_p(D)\gg\sum_{p\leqslant \log_2k}p\sim\frac{(\log_2 k)^2}{2\log\log_2 k}.\]Consequently, $\log N\gg (\log k)^3/\log\log k$, from which Theorem \ref{thm:2} easily follows.

\section*{Acknowledgments}

I would like to express my deep gratitude to Andrew Granville, for his support and for his contribution \cite{GR} with the material presented throughout Sections \ref{sec:3} and \ref{sec:4}. Additionally, I would like to thank Alexandru Gica for his constant guidance, his insightful comments, and for proofreading my paper.

\bibliography{main}
\bibliographystyle{aomalpha}

\end{document}